\documentclass{amsart}
\usepackage{amsmath,amsthm,amssymb,cite}
\usepackage{mathtools}
\usepackage{mathrsfs}
\usepackage{color}
\usepackage{verbatim}
\mathtoolsset{showonlyrefs}
\usepackage{enumitem}

\newtheorem{theorem}{Theorem}
\newtheorem{corollary}[theorem]{Corollary}
\newtheorem{lemma}[theorem]{Lemma}
\newtheorem{proposition}[theorem]{Proposition}

\newcommand{\omr}[1]{\operatorname{R}_{#1}\!}

\newcommand{\sln}{\operatorname{SL}(n)}
\newcommand{\gln}{\operatorname{GL}(n)}

\newcommand{\R}{\mathbb R}
\newcommand{\B}{{B^n}}
\newcommand{\sn}{{\mathbb S}^{n-1}}

\newcommand{\Beta}{\operatorname{B}}

\renewcommand{\d}{\,\mathrm{d}}

\newcommand{\vol}[1]{\vert #1\vert}
\newcommand{\pl}[1]{#1_{\scriptscriptstyle +}}
\newcommand{\mn}[1]{#1_{\scriptscriptstyle -}}

\renewcommand{\chi}{\operatorname{1}}

\newcommand{\hls}[1]{\operatorname{S}_{#1}\!} 

\title[Affine Hardy--Littlewood--Sobolev inequalities]{Affine Hardy--Littlewood--Sobolev inequalities}

\author{Juli\'an Haddad}
\address{Departamento de An\'alisis Matem\'atico, Facultad de Matem\'aticas, Universidad de Sevilla, Sevilla, Spain}
\email{jhaddad@us.es}

\author{Monika Ludwig}
\address{Institut f\"ur Diskrete Mathematik und Geometrie,
Technische Universit\"at Wien,
Wiedner Hauptstra\ss e 8-10/1046,
1040 Wien, Austria}
\email{monika.ludwig@tuwien.ac.at}

\begin{document}

\begin{abstract}
Sharp affine Hardy--Littlewood--Sobolev inequalities for functions on $\R^n$ are
established, which are significantly stronger than (and directly imply) the sharp Hardy--Littlewood--Sobolev inequalities by Lieb and by Beckner, Dou, and Zhu. In addition, sharp reverse inequalities for the new inequalities and the affine fractional $L^2$ Sobolev inequalities are obtained for log-concave functions on $\R^n$.

\end{abstract}

\maketitle

\section{Introduction}
Lieb \cite{Lieb83} established the following sharp Hardy--Littlewood--Sobolev inequality  (HLS inequality): 
\begin{equation}\label{eq_HLS}
\gamma_{n,\alpha}\Vert f\Vert_{\frac{2n}{n+\alpha}}^2\ge 
\int_{\R^n}\int_{\R^n} \frac{f(x) f(y)}{\vert x-y\vert^{n-\alpha}} \d x\d y
\end{equation}
for $0<\alpha<n$ and non-negative $f\in L^p(\R^n)$  with $p = {2n}/({n+\alpha})$.
There is equality if and only if $f(x) = a (1 + \lambda\,|x-x_0|^2)^{-(n+\alpha)/2}$  for $x\in\R^n$ with $a\ge 0$, $\lambda>0$, and $x_0\in\R^n$.
Here $\vert \cdot\vert$ denotes  the Euclidean norm in $\R^n$ and $\Vert f\Vert_{p}^p =\int_{\R^n} \vert f(x)\vert^p\d x$ while $L^p(\R^n)$ is the space of measurable functions $f:\R^n\to\R$ with $\Vert f\Vert_{p}<\infty$. 
The constant is given by
\begin{equation}\label{eq_HLS_const}
    \gamma_{n,\alpha}= \pi^{\frac{n-\alpha}2}\frac{\Gamma(\frac{\alpha}2)}{\Gamma(\frac{n+\alpha}2)}\Big(\frac{\Gamma(n)}{\Gamma(\frac n2)}\Big)^{\frac \alpha n},
\end{equation}
where $\Gamma$ is the gamma function. The HLS inequality \eqref{eq_HLS} can be considered as a weak Young inequality and is equivalent by duality to the sharp fractional $L^2$ Sobolev inequality (for more information, see \cite{Lieb:Loss, CarlenLoss, Carlen2017, FrankLieb2010}).

The HLS inequalities are invariant under translations, rotations, and inversions but not under volume-preserving linear transformations. For geometric questions, affine inequalities, that is, inequalities that are unchanged under translations and volume-preserving linear transformations turned out to be very powerful. 
The best-known example is the Petty projection inequality for convex bodies (that is, compact convex sets with non-empty interior) in $\R^n$, which is stronger than the Euclidean isoperimetric inequality and directly implies it (see, for example, \cite{Schneider:CB2, Gardner}, and see \cite{LYZ2000,Haberl:Schuster1, Zhang99, Tuo_Wang} for related results in the $L^p$ Brunn--Minkowski theory and for more general sets). 
Very recently, Milman and Yehudayoff \cite{MilmanYehudayoff} established isoperimetric inequalities for affine quermassintegrals, which are significantly stronger than the isoperimetric inequality for the classical quermassintegrals (thereby confirming a conjecture by Lutwak \cite{Lutwak88b}).
Gaoyong Zhang's affine Sobolev inequality \cite{Zhang99} is an affine version of the classical $L^1$ Sobolev inequality. 
It was extended to functions of bounded variation by Tuo Wang \cite{Tuo_Wang}, and corresponding results for $L^p$ Sobolev inequalities were established by Lutwak, Yang, and Zhang \cite{LYZ2002b} and Haberl and Schuster \cite{Haberl:Schuster2}. 
Fractional Petty projection inequalities were recently obtained by the authors in \cite{HaddadLudwig_fracsob} and affine fractional $L^p$ Sobolev inequalities  in \cite{HaddadLudwig_fracsob, HaddadLudwig_Lpfracsob}. In all cases, the affine inequalities are significantly stronger than (and imply) their Euclidean counterparts.

\goodbreak
The main aim of this paper is to establish sharp affine HLS inequalities that are stronger than Lieb's sharp HLS inequalities \eqref{eq_HLS}. 

\begin{theorem}\label{thm_aHLS}
For $\,0<\alpha<n$ and non-negative $f\in L^{{2n}/({n+\alpha})}(\R^n)$,
\begin{multline}
\gamma_{n,\alpha} \Vert f\Vert_{\frac{2n}{n+\alpha}}^2
\ge
(n \omega_n)^{1-\frac \alpha n}\Big( 
\int_{\sn}\big(\int_0^\infty t^{\alpha-1} \int_{\R^n} f(x)f(x+t\xi) \d x\d t\big)^{\frac n \alpha}\d \xi\Big)^{\frac \alpha n} \\
\ge
\int_{\R^n}\int_{\R^n} \frac{f(x) f(y)}{\vert x-y\vert^{n-\alpha}} \d x\d y.
\end{multline}
There is equality in the first inequality precisely if $f(x) = a (1 +|\phi (x-x_0)|^2)^{-(n+\alpha)/2}$ for $x\in\R^n$ with $a\ge 0$, $\phi\in\gln$, and $x_0\in\R^n$.
There is equality in the second inequality if $f$ is radially symmetric.
\end{theorem}

\noindent
Here, $\sn$ is the unit sphere in $\R^n$ and integration on $\sn$ is with respect to the $(n-1)$-dimensional Hausdorff measure while $\omega_n$ is the $n$-dimensional volume of the $n$-dimensional unit ball.

\goodbreak
To prove Theorem \ref{thm_aHLS}, for given $0<\alpha<n$, we introduce the star-shaped set $\hls \alpha f$ associated to $f$, defined by its radial function for $\xi\in\sn$ as
\begin{equation}\label{eq_defhls}
\rho_{\hls{\alpha}f}(\xi)^{\alpha}= \int_0^\infty t^{\alpha -1} \int_{\R^n} f(x)f(x+t\xi) \d x \d t
\end{equation}
(see Section \ref{sec_hls_body} for details). The first inequality from Theorem \ref{thm_aHLS}  now can be written as
\begin{equation}\label{eq_ahls}
\gamma_{n,\alpha} \Vert f\Vert_{\frac{2n}{n+\alpha}}^2
\ge
n \omega_n^{{1-\frac \alpha n}} \vol{\hls \alpha f}^{\frac \alpha n},
\end{equation}
where $\vol{\cdot}$ denotes $n$-dimensional Lebesgue measure. Since both sides of \eqref{eq_ahls} are invariant under translations of $f$ and 
\[ \hls{\alpha} \,(f\circ\phi^{-1}) = \phi \hls{\alpha} f\]
for volume-preserving linear transformations $\phi: \R^n\to \R^n$,
it follows that inequality \eqref{eq_ahls} is indeed an affine inequality. 

As a critical tool in the proof of Theorem \ref{thm_aHLS}, we introduce an anisotropic version of the right side of \eqref{eq_HLS}. In addition, we use the Riesz rearrangement inequality, its equality case due to Burchard \cite{Burchard96}, and  the original HLS inequalities for radially symmetric functions. 

Recently, Dou and Zhu \cite{DouZhu15} and Beckner \cite{Beckner2015} obtained sharp HLS inequalities also  for $\alpha>n$. In Section \ref{sec_reverse_HLS}, we will establish sharp affine HLS inequalities for $\alpha>n$, which are stronger than (and directly imply) their results.

In Section \ref{sec_GZ}, for $E\subset \R^n$ measurable, we consider $\hls \alpha \chi_E$, where $\chi_E$ is the indicator function of $E$. For a convex body $E\subset \R^n$, we show that $\hls \alpha \chi_E$ is proportional to the radial $\alpha$-mean body of $E$, an important notion that was introduced by Gardner and Zhang \cite{GZ}. Sharp isoperimetric inequalities for radial $\alpha$-mean bodies were recently obtained in \cite{HaddadLudwig_fracsob}. Sharp reverse inequalities were already established by Gardner and Zhang \cite{GZ}. They generalize Zhang's reverse Petty projection inequality \cite{Zhang91}, and equality is attained precisely for $n$-dimensional simplices.  

\goodbreak
In Section \ref{sec_reverse}, we establish reverse inequalities also in the functional setting.

\begin{theorem}\label{thm_reverseineq}
For $\,0<\alpha<n$ and  log-concave $f\in L^2(\R^n)$, 
\[ \frac{\Gamma(n+1)^\frac{\alpha}n}{\Gamma(\alpha)}\,\vol{\hls\alpha f}^\frac{\alpha}n \geq  \|f\|_2^{2-\frac{2\alpha}n} \|f\|_1^{\frac{2\alpha}n} \geq \Vert f\Vert_{\frac{2n}{n+\alpha}}^2.\]
There is equality in the first inequality if $f(x)=a\,e^{-\Vert x - x_0\Vert_\Delta}$ for $x\in\R^n$ with $a\ge 0$, $x_0\in\R^n$, and $\Delta$ is an $n$-dimensional simplex having a vertex at the origin. 
\end{theorem}

\noindent
Here, $\Vert\cdot\Vert_\Delta$ is the gauge function of $\Delta\subset \R^n$. The second inequality follows from H\"older's inequality; for the proof of the first inequality, see Section \ref{sec_reverse}.

\goodbreak
We remark that it is easy to see that for general, non-negative $f\in L^2(\R^n)$, no non-trivial reverse inequality can hold.  In Section \ref{sec_reverse}, we establish reverse inequalities also for $\alpha>n$ (see Theorem \ref{thm_reverseineq2}) and obtain results for $s$-concave functions for $s>0$. Moreover, we will establish reverse affine fractional $L^2$  Sobolev inequalities.

\section{Preliminaries}\label{sec_prelim}

We collect results on symmetrization, star-shaped sets, log-concave and $s$-concave functions, and fractional polar projection bodies. 

\subsection{Symmetrization} 

Let $E \subseteq \R^n$ be a Borel set of finite Lebesgue measure.
The Schwarz symmetral of $E$, denoted by $E^\star$, is the closed, centered Euclidean ball with the same volume as $E$.

Let $f$ be a non-negative measurable function with superlevel sets of finite measure. Let $\{f\ge t\}=\{ x\in \R^n: f(x)\ge t\}$ for $t\in\R$.  We say that $f$ is non-zero if $\{f\ne 0\}$ has positive measure, and we identify functions that are equal up to a set of measure zero.
The layer cake formula states that
\begin{equation}\label{eq_layer_cake}
f(x) = \int_0^\infty \chi_{\{f\geq t\}}(x) \d t
\end{equation}
for almost every $x \in \R^n$ and allows us to recover the function from its superlevel sets. Here, for $E\subset \R^n$, the indicator function $\chi_E$ is defined by $\chi_E(x)=1$ for $x\in E$ and $\chi_E(x)=0$ otherwise.

The Schwarz symmetral of $f$, denoted by $f^\star$, is defined as
\begin{equation}\label{eq_layer_cake_schwarz}
f^\star(x) = \int_0^\infty \chi_{\{f\geq t\}^\star}(x) \d t
\end{equation}
for $x\in\R^n$.
Hence $f^\star$ is determined a.e.\ by the properties of being radially symmetric and having superlevel sets of the same measure as those of $f$. Note that $f^\star$ is also called the symmetric decreasing rearrangement of $f$. We say that $f^\star$ is strictly symmetric decreasing, if $f^\star(x)>f^\star(y)$ whenever $\vert x\vert <\vert y\vert$.

Let $f,g$ be non-negative measurable functions with superlevel sets of finite measure such that $f \leq g$ in $\R^n$. Then, clearly, $\chi_{\{f\geq t\}^\star} \leq \chi_{\{g\geq t\}^\star}$ for all $t>0$ and, consequently, $f^\star \leq g^\star$.

\goodbreak
The proofs of our results make use of the Riesz rearrangement inequality (see, for example, \cite[Theorem 3.7]{Lieb:Loss}). 

\begin{theorem}[Riesz's rearrangement inequality]
        \label{thm_BLL}
        For $f,g,k:\R^n \to \R$ non-negative, measurable functions with superlevel sets of finite measure,
         \[ \int_{\R^n}\int_{\R^n}  f(x) k(x-y) g(y) \d x \d y \leq \int_{\R^n}\int_{\R^n}  f^\star(x) k^\star(x-y) g^\star(y) \d x \d y. \]
\end{theorem}

\goodbreak
We will use the characterization of equality cases of the Riesz rearrangement inequality due to Burchard  \cite{Burchard96}.

\begin{theorem}[Burchard]
        \label{thm_burchard}
        Let $A,B$ and $C$ be sets of finite positive measure in $\,\R^n$ and denote by $\alpha, \beta$ and $\gamma$ the radii of their Schwarz symmetrals $A^\star, B^\star$ and $C^\star$.
        For $\,\vert\alpha - \beta\vert < \gamma < \alpha+\beta$, there is equality in
        \[ \int_{\R^n}\int_{\R^n} \chi_A(y) \chi_B(x-y) \chi_C(x) \d x \d y \leq  \int_{\R^n}\int_{\R^n} \chi_{A^\star}(y) \chi_{B^\star}(x-y) \chi_{C^\star}(x) \d x \d y \]
        if and only if, up to sets of measure zero,
        \[A=a+\alpha D,\, B = b+\beta D,\, C = c+\gamma D,\]
        where $D$ is a centered ellipsoid, and $a,b$ and $c=a+b$ are vectors in $\R^n$.
\end{theorem}

\begin{theorem}[Burchard]
    \label{thm_burchard2}
    Let $f,g,k:\R^n \to \R$ non-negative, non-zero, measurable functions with superlevel sets of finite measure such that
    \begin{equation}
    \int_{\R^n}\int_{\R^n}  f(x) k(x-y) g(y) \d x \d y <\infty.
    \end{equation}
    If at least two of the Schwarz symmetrals $f^\star, g^\star, k^\star$ are strictly symmetric decreasing, then
    there is equality in
    \begin{equation}
          \int_{\R^n}\int_{\R^n}  f(x) k(x-y) g(y) \d x \d y \leq \int_{\R^n}\int_{\R^n}  f^\star(x) k^\star(x-y) g^\star(y) \d x \d y
    \end{equation}
    if and only if there is a volume-preserving $\phi\in\gln$ and $a,b,c\in\R^n$ with $c=a+b$ such that 
    \[f(x)=f^\star (\phi^{-1} x-a),\, k(x)=k^\star (\phi^{-1}x-b), g(x)=g^\star (\phi^{-1} x-c)\]
    for $x\in\R^n$.
\end{theorem}

\subsection{Star-shaped sets and dual mixed volumes}
A closed set $K \subseteq \R^n$ is star-shaped (with respect to the origin) if the interval $[0,x]\subset K$ for every $x\in K$. The gauge function $\|\cdot\|_K : \R^n \to [0,\infty]$ of a star-shaped set is defined as
\[\|x\|_K = \inf\{ \lambda > 0 : x \in \lambda K\}\]
and the radial function $\rho_K:\R^n \setminus\{0\} \to [0,\infty]$ as
\[\rho_K(x) = \|x\|_K^{-1} = \sup\{\lambda \geq 0:\lambda x \in K\}.\]
For the $n$-dimensional unit ball $\B$, we have $\Vert \cdot \Vert_{\B}=\vert\cdot \vert$. The $n$-dimensional Lebesgue measure or volume of a star-shaped set $K\subset\R^n$ with measurable radial function is given by
\[\vol{K}= \frac1 n \int_{\sn} \rho_K(\xi)^n\d \xi.\]
We call $K$ a star body if its radial function is strictly positive and continuous in $\R^n \setminus \{0\}$. 

Let $\alpha\in\R\backslash\{0,n\}$. For star-shaped sets $K, L \subseteq \R^n$  with measurable radial functions, the dual mixed volume is defined as
\[\tilde V_\alpha(K,L) = \frac 1n \int_{\sn} \rho_K(\xi)^{n-\alpha} \rho_L(\xi)^{\alpha} \d \xi.\]
Note that 
\[\tilde V_\alpha(K,K) = \vol{K}.\]
For $0<\alpha<n$ and star-shaped sets $K, L \subseteq \R^n$ of finite volume, the dual mixed volume inequality states that 
\begin{equation}
    \label{eq_mixedvolume}
    \tilde V_{\alpha}(K,L) \leq \vol{K}^{({n-\alpha})/n} \vol{L}^{ \alpha/n}.
\end{equation}
Equality holds if and only if $K$ and $L$ are dilates, where we say that star-shaped sets $K$ and $L$ are dilates if $\rho_K=c\,\rho_L$ almost everywhere on $\sn$ for some $c\ge 0$. The definition of dual mixed volume for star bodies is due to Lutwak \cite{Lutwak75}, where also the dual mixed volume inequality \eqref{eq_mixedvolume} is derived from H\"older's inequality. For $0<\alpha<n$ and star-shaped sets of finite volume, it follows from \eqref{eq_mixedvolume} that the dual mixed volume is finite.
For star-shaped sets $K,L \subseteq \R^n$ and $\alpha>n$, the dual mixed volume inequality states that 
\begin{equation}
    \label{eq_mixedvolumere}
    \tilde V_{\alpha}(K,L) \geq \vol{K}^{({n-\alpha})/n} \vol{L}^{ \alpha/n}.
\end{equation}
It follows from the equality case of H\"older's inequality that equality holds for finite $\tilde V_{\alpha}(K, L)$ if and only if $K$ and $L$ are dilates. See \cite{Schneider:CB2, Gardner} for more information on dual mixed volumes.

\subsection{Log-concave and $s$-concave functions}
A function $f: \R^n\to [0,\infty)$ is log-concave, if $x\mapsto \log f(x)$ is a concave function on $\R^n$ with values in $[-\infty, \infty)$. 
For $s>0$, a function $f: \R^n\to [0,\infty)$ is $s$-concave, if $f^s$ is concave on its support.

The following result is a consequence of the Pr\'ekopa--Leindler inequality. It can be found as Theorem 11.3 in \cite{Gardner_Gorizia}.  Log-concave functions are included as the case $s=0$.
\begin{lemma}\label{lem_GG}
Let $s\ge 0$. If $f, g\in L^1(\R^n)$ are $s$-concave, then their convolution is  $s/(ns+2)$-concave.
\end{lemma}

\noindent
Here, the convolution of $f,g\in L^1(\R^n)$ is defined by
\[f*g(x)=\int_{\R^n} f(x-y)\,g(y)\d y\]
for $x\in\R^n$. We define $f^-$ by $f^-(x)=f(-x)$ for $x\in\R^n$ and will often consider the function 
\begin{equation}\label{auto}
    f*f^-(x)=\int_{\R^n} f(x)\,f(x+y)\d y
\end{equation}
for $x\in\R^n$.

We also consider the limiting case $s=\infty$. We say that a function $f:\R^n\to[0,\infty)$ is $s$-concave with $s=\infty$ if it is a multiple of the indicator function of a convex body.

\subsection{Fractional $L^2$ polar projection bodies}

For measurable $f:\R^n\to \R$ and $0<\alpha<1$, the $\alpha$-fractional $L^2$ polar projection body of $f$, denoted by $\Pi_2^{\ast, \alpha} f$, was defined in \cite{HaddadLudwig_Lpfracsob} by its radial function 
for $\xi\in\sn$ as
\begin{equation}\label{eq_defpp}
\rho_{\Pi_2^{\ast, \alpha} f}(\xi)^{-2\alpha}= \int_0^\infty t^{-2 \alpha-1}\, \int_{\R^n}\vert f(x +t\xi)-f(x)\vert^2 \d x\d t.
\end{equation}
The fractional Sobolev space $W^{\alpha,2}(\R^n)$ is the set of all $f\in L^2(\R^n)$ such that
\[\int_{\R^n}\int_{\R^n} \frac{\vert f(x)-f(y)\vert^2}{\vert x-y\vert^{n+2\alpha}}\d x\d y<\infty.\]
The affine fractional $L^2$ Sobolev inequality \cite[Theorem 1]{HaddadLudwig_Lpfracsob} states  that
\begin{equation}\label{eq_L2fracsob}
\Vert f\Vert_{{\frac {2n }{n-2\alpha}}}^2
\le
\sigma_{n,2,\alpha}\, n\omega_n^{1+\frac{2\alpha}n} \vol{\Pi_2^{\ast, \alpha} f}^{-\frac {2\alpha}n}
\end{equation}
for $f\in W^{\alpha,2}(\R^n)$ and $0<\alpha<1$, where $\sigma_{n,2,\alpha}$ is an explicitly known constant.

\goodbreak
\section{The Star-shaped Set $\hls \alpha f$}\label{sec_hls_body}

Let $f: \R^n\to [0,\infty)$ be measurable, $K\subset \R^n$  star-shaped with measurable radial function and $\alpha>0$.
Anisotropic fractional Sobolev norms were introduced in \cite{Ludwig:fracperi, Ludwig:fracnorm} and used in \cite{HaddadLudwig_fracsob}. Here, we consider 
\[\int_{\R^n}\int_{\R^n} \frac{f(x) f(y)}{\|x-y\|_K^{n-\alpha}}\d x\d y,\]
an anisotropic version of the functional from \eqref{eq_HLS}.
Using Fubini's theorem, polar coordinates, and \eqref{eq_defhls}, we obtain that
\begin{align*}
	\int_{\R^n}\int_{\R^n} \frac{f(x) f(y)}{\|x-y\|_K^{n-\alpha}}\d x\d y
	&= \int_{\R^n}\int_{\R^n} \frac{ f(y) f(y+z)}{\|z\|_K^{n-\alpha}} \d y \d z \\
	&= \int_{\sn} \int_0^\infty \rho_K(t \xi )^{n-\alpha}\,  t^{n-1} \int_{\R^n} f(y) f(y+t \xi)  \d y \d t \d\xi \\
	&= \int_{\sn} \rho_K(\xi )^{n-\alpha} \int_0^\infty t^{\alpha-1} \int_{\R^n} f(y) f(y+t \xi) \d y \d t \d\xi \\
	&= \int_{\sn} \rho_K(\xi )^{n-\alpha} \rho_{\hls\alpha f}(\xi)^\alpha d\xi.
\end{align*}
Hence,
\begin{equation}\label{eq_new_def_c}
\int_{\R^n}\int_{\R^n} \frac{f(x) f(y)}{\|x-y\|_K^{n-\alpha}}\d x\d y=n \tilde V_\alpha (K, \hls\alpha f)
\end{equation}
for measurable $f:\R^n\to [0,\infty)$ and star-shaped $K\subset\R^n$ with measurable radial function.

Note that, using polar coordinates and Fubini's theorem, we obtain that
\begin{align}
\begin{split}\label{eq_hlsn}
    \vol{\hls nf}
    &=\frac 1n \int_{\sn} \rho_{\hls nf}(\xi)^n \d \xi  \\
    &=\frac 1n \int_{\sn} \int_0^\infty \int_{\R^n}t^{n-1} f(x) f(x+t\xi) \d x \d t \d \xi \\
    &=\frac 1n \int_{\R^n} \int_{\R^n} f(x) f(x+y) \d y \d x  \\
    &=\frac 1n \|f\|_1^2
\end{split}
\end{align}
for measurable $f:\R^n\to [0,\infty)$. 
\goodbreak
We remark that for $0< \alpha <n$ and given measurable $f:\R^n\to[0,\infty)$, the dual mixed volume inequality \eqref{eq_mixedvolume}  and \eqref{eq_new_def_c} imply that
\begin{equation}
\sup\big\{ \int_{\R^n}\int_{\R^n} \frac{f(x) f(y)}{\|x-y\|_K^{n-\alpha}}\d x\d y: K\subset \R \text{ star-shaped}, \vol{K}=\omega_n\big\} = n\omega_n^{1-\frac \alpha n} \vol{\hls \alpha f}^\frac{\alpha}n
\end{equation}
is attained precisely for a suitable dilate of $\hls \alpha f$ if $\vol{\hls \alpha f}$ is finite. In this sense, $\hls \alpha f$ is the optimal choice of the star-shaped set $K$ for given $f$. For $\alpha >n$ and given measurable $f:\R^n\to[0,\infty)$, 
the dual mixed volume inequality \eqref{eq_mixedvolumere}  and \eqref{eq_new_def_c} imply that
\begin{equation}
\inf\big\{ \int_{\R^n}\int_{\R^n} \frac{f(x) f(y)}{\|x-y\|_K^{n-\alpha}}\d x\d y: K\subset \R \text{ star-shaped}, \vol{K}=\omega_n\big\} =n\omega_n^{1-\frac \alpha n} \vol{\hls \alpha f}^\frac{\alpha}n
\end{equation}
is attained precisely for a suitable dilate of $\hls \alpha f$ if $\vol{\hls \alpha f}$ is finite. Again, $\hls \alpha f$ is the optimal choice in this sense.

\goodbreak
We mention the following property of $\hls \alpha f$  for log-concave $f$.

\begin{proposition}\label{prop_convex}
If $f:\R^n\to [0,\infty)$ is log-concave and in $L^1(\R^n)$, then $\hls \alpha f$ is a convex body for every $\alpha>0$.
\end{proposition}

\begin{proof}
First, observe that $f*f^-$ is in $L^1(\R^n)$ due to Fubini's theorem.
By Lemma \ref{lem_GG}, it is also log-concave.
Thus we may apply the well-known result of Ball \cite[Theorem 5]{Ball1988} (or see \cite[Corollary 4.2]{GZ}), saying that if a non-negative function $g$ is even, log-concave and integrable, then the function
\[\xi \mapsto \left(\int_0^\infty t^{\alpha-1} g(t\xi) \d t\right)^{-1/\alpha}\]
is a norm for every $\alpha>0$.
\end{proof}

\section{Affine HLS Inequalities for $0<\alpha<n$}

The following result is a consequence of the Riesz rearrangement inequality and its equality case from Theorem \ref{thm_burchard2}.

\begin{lemma}\label{lem_strict_rearrangement}
Let $q>0$ and $K\subset\R^n$ a star-shaped set with measurable radial function and $\vol{K}>0$.
For a non-zero, measurable function $f:\R^n \to [0,\infty)$ such that
\begin{equation}
\int_{\R^n}\int_{\R^n} \frac{f(x) f(y)}{\Vert x-y\Vert_K^q} \d x \d y<\infty
\end{equation}
and  strictly symmetric decreasing $f^\star$, 
there is equality in
\begin{equation}\label{eq_riesz}
       \int_{\R^n}\int_{\R^n}  \frac{f(x) f(y)}{\Vert x-y\Vert_K^q} \d x \d y \leq \int_{\R^n}\int_{\R^n} \frac{f^\star(x) f^\star(y)}{\Vert x-y\Vert_{K^\star}^q} \d x \d y
\end{equation}
if and only if $K$ is a centered ellipsoid and $f$ is a  translate of $f^\star$.
\end{lemma}
\begin{proof}
The inequality \eqref{eq_riesz} follows from the Riesz rearrangement inequality, since $\rho_{K^\star}^q = ( \rho_K^q)^\star$. This follows directly from the definition \eqref{eq_layer_cake_schwarz} by observing that $\rho_K(y)^q \geq t$ if and only if $y \in t^{-1/q} K$, and consequently,
    \begin{align}
        (\rho_K^q)^\star(x)
        &= \int_0^\infty \chi_{\{\rho_K^q \geq t\}^\star} (x) \d t \\
        &= \int_0^\infty \chi_{t^{-1/q} K^\star} (x) \d t \\
        &= \rho^q_{K^\star}(x),
    \end{align}
for $x\in\R^n\backslash\{0\}$. For the equality case, we use Theorem \ref{thm_burchard2} with $f=g$. The equality in \eqref{eq_riesz} implies that $b=0$, $\rho_K=\rho_{K^\star}\circ \phi^{-1}$ for a volume-preserving $\phi\in\gln$ and that $f$ is a translate of $f^\star$. Since $K^\star$ is a ball, this concludes the proof. 
\end{proof}

\goodbreak
We require the following lemmas for the proof of Theorem \ref{thm_aHLS}.

\begin{lemma}\label{lem_smooth_compact_support_Riesz}
Let  $\,0<\alpha<n$ and $K\subset\R^n$ be star-shaped with measurable radial function. If  $
f:\R^n\to[0,\infty)$ is non-zero and measurable
 and
\begin{equation}
\int_{\R^n}\int_{\R^n} \frac{f(x) f(y)}{\Vert x-y\Vert_K^{n-\alpha}} \d x \d y<\infty,
\end{equation}
then
\begin{equation}
\tilde V_\alpha(K, \hls\alpha f) \leq \tilde V_\alpha(K^\star, \hls\alpha f^\star).
\end{equation}
For $\vol{K}>0$ and $f^\star$ strictly symmetric decreasing, there is equality if and only if $K$ is a centered ellipsoid and $f$ is a translate of $f^\star$.
\end{lemma}

\goodbreak
\begin{proof}
By \eqref{eq_new_def_c} and the Riesz rearrangement inequality, Theorem \ref{thm_BLL}, we have
\[\tilde V_\alpha(K, \hls\alpha f) \leq \tilde V_\alpha(K^\star, \hls\alpha f^\star).\]
By Lemma \ref{lem_strict_rearrangement}, there is equality if and only if $K$ is a centered ellipsoid and $f$ is a translate of $f^\star$.
\end{proof}

\goodbreak
\begin{lemma}\label{lem_new_rearrangement}
Let $0<\alpha<n$ and $p=2\alpha/(n+\alpha)$.
For non-negative $f\in L^p(\R^n)$, 
	\begin{equation}
	\vol{\hls\alpha f} \leq \vol{\hls\alpha f^\star}.
	\end{equation}
For $f^\star$ strictly symmetric decreasing with $\vol{\hls \alpha f^\star}<\infty$,
there is equality if and only if $f$ is a translate of $f^\star$.
\end{lemma}

\begin{proof}
First, assume that $\,\vol{\hls \alpha f}<\infty$.  By Lemma \ref{lem_smooth_compact_support_Riesz} with $K=\hls \alpha f$ and the dual mixed volume inequality \eqref{eq_mixedvolume} for $0<\alpha<n$,  we have
\begin{align}
	\vol{\hls\alpha f}
	&= \tilde V_\alpha(\hls \alpha f, \hls\alpha f)\\ 
	&\leq \tilde V_\alpha((\hls\alpha f)^\star, \hls\alpha f^\star)\\
	&\leq \vol{(\hls\alpha f)^\star}^{1-\frac {\alpha}n} \vol{\hls\alpha f^\star}^{\frac \alpha n}\\
	&= \vol{\hls \alpha f}^{1-\frac {\alpha}n} \vol{\hls\alpha f^\star}^{\frac \alpha n}.
\end{align}
The equality case follows from Lemma \ref{lem_smooth_compact_support_Riesz}. 

\goodbreak
Second, assume that $\vol{\hls\alpha f} = \infty$. For $k\ge1$, define 
\[f_{(k)}(x) = f(x) \chi_{k \B}(x).
\]
Note that $f_{(k)}$ is non-decreasing with respect to $k$ and converges to $f$ pointwise. 
By the monotone convergence theorem, we obtain
\begin{equation}
\lim_{k\to\infty}\int_0^\infty t^{\alpha-1} \int_{\R^n} f_{(k)}(x)f_{(k)}(x+t \xi) \d x \d t= \int_0^\infty t^{\alpha-1} \int_{\R^n} f(x)f(x+t \xi) \d x \d t
\end{equation}
and the convergence is monotone. A second application of the monotone convergence theorem shows that
\begin{multline}
\lim_{k\to\infty}\int_{\sn} \left( \int_0^\infty t^{\alpha-1} \int_{\R^n} f_{(k)}(x)f_{(k)}(x+t \xi) \d x \d t \right)^{\frac n\alpha} \d\xi \\= \int_{\sn} \left( \int_0^\infty t^{\alpha-1} \int_{\R^n} f(x)f(x+t \xi) \d x \d t\right)^{\frac n\alpha} \d\xi. 
\end{multline}
Hence,
\begin{equation}\label{eq_volume_converge}
\lim_{k\to\infty}\vol{\hls \alpha f_{(k)}} = \vol{\hls \alpha f} = \infty.
\end{equation}
Since $f\in L^p(\R^n)$, the function $f_{(k)}$ has compact support and $\vol{\hls \alpha f_{(k)}}<\infty$ for $k\ge 1$.
Since $f_{(k)} \leq f$ in $\R^n$, we deduce that $(f_{(k)})^\star \leq f^\star$ in $\R^n$ too, so the first part of the lemma implies that
\begin{equation}
\vol{\hls\alpha f_{(k)}} \leq \vol{\hls\alpha (f_{(k)})^\star} \leq \vol{\hls\alpha f^\star} 
\end{equation}
for $k\ge 1$.
It follows from \eqref{eq_volume_converge} that
 $\vol{\hls \alpha f_{(k)}} \to \infty$, which shows that
		\begin{equation}
		\vol{\hls\alpha f^\star} = \infty,
		\end{equation}
which is what we wanted to show.
\end{proof}

\subsection{Proof of Theorem \ref{thm_aHLS}}
Since $\|f\|_p=\|f^\star\|_p$, we obtain by using  the classical HLS inequality \eqref{eq_HLS}, equation \eqref{eq_new_def_c}, the equality case of \eqref{eq_mixedvolume} and the fact that $\hls\alpha f^\star$ is a ball, and Lemma \ref{lem_new_rearrangement} that
\begin{align}
	\gamma_{n,\alpha} \|f\|_p^2 
	&\geq \int_{\R^n}\int_{\R^n} \frac{f^\star(x)f^\star(y)}{|x-y|^{n-\alpha}} \d x \d y \\
	&= n \tilde V_\alpha (\B, \hls \alpha f^\star)\\
	&= n \omega_n^{{1-\frac \alpha n}} \vol{\hls\alpha f^\star}^{\frac\alpha n}\\
	&\geq n \omega_n^{{1-\frac \alpha n}} \vol{\hls\alpha f}^{\frac\alpha n}.
\end{align}
If there is equality throughout, then $f^\star$ realizes equality in the HLS inequality \eqref{eq_HLS}. Hence $f^\star(x) = a (1 + \lambda\, | x|^2)^{-n/p}$ for some $a\ge0$ and $\lambda>0$. Consequently, $f^\star$ is strictly symmetric decreasing, and we may apply Lemma \ref{lem_new_rearrangement} to obtain the equality case in the first inequality in Theorem \ref{thm_aHLS}.

For the second inequality, we set $K=\B$ in \eqref{eq_new_def_c} and apply the dual mixed volume inequality \eqref{eq_mixedvolume} to obtain
\begin{align}
    \int_{\R^n}\int_{\R^n} \frac{f(x)f(y)}{\vert x-y\vert^{n-\alpha}}\d x\d y
    &= n \tilde V_{\alpha}(\B, \hls \alpha f)
    \leq n \omega_n^{1-\frac{\alpha} n} \vol{\hls \alpha f}^{\frac\alpha n}.
\end{align}
There is equality precisely if $\hls \alpha f$ is a ball, which is the case for radially symmetric functions.

\section{Affine HLS Inequalities for $\alpha>n$}\label{sec_reverse_HLS}

Jingbo Dou and Meijun Zhu \cite{DouZhu15}  and William Beckner \cite{Beckner2015} established sharp HLS inequalities for $\alpha>n$ (also see \cite{NgoNguyen17, carrillo2019reverse}):
\begin{equation}\label{eq_HLS2}
\gamma_{n,\alpha}\Vert f\Vert_{\frac{2n}{n+\alpha}}^2\le 
\int_{\R^n}\int_{\R^n} \frac{f(x) f(y)}{\vert x-y\vert^{n-\alpha}} \d x\d y
\end{equation}
for non-negative $f\in L^p(\R^n)$  with $p = {2n}/({n+\alpha})$, where $\gamma_{n,\alpha}$ is defined in \eqref{eq_HLS_const}. There is equality if $f(x) = a (1 + \lambda\,|x-x_0|^2)^{-(n-\alpha)/2}$ for $x\in\R^n$ with $a\ge 0$, $\lambda>0$ and $x_0\in\R^n$.

We will establish sharp affine  HLS inequalities for $\alpha>n$ that strengthen and imply \eqref{eq_HLS2}. 
We require the following lemmas.

\goodbreak
The following result is a consequence of the Riesz rearrangement inequality and Theorem \ref{thm_burchard}. Note that the middle function in \eqref{eq_HLSr} has superlevel sets of infinite measure.

\begin{lemma}\label{lem_reverse_strict_rearrangement}
Let $q>0$ and $K\subset\R^n$ a star-shaped set with $0<\vol{K}<\infty$.
For non-negative, non-zero $f\in L^p(\R^n)$ such that
\begin{equation}\label{eq_HLSr}
\int_{\R^n}\int_{\R^n} {f(x) {\Vert x-y\Vert_K^q}\,  f(y)}\d x \d y<\infty,
\end{equation}
there is equality in
\begin{equation}\label{eq_riesz_equality}
       \int_{\R^n}\int_{\R^n} {f(x) {\Vert x-y\Vert_K^q} \,f(y)} \d x \d y \geq \int_{\R^n}\int_{\R^n} {f^\star(x) {\Vert x-y\Vert_{K^\star}^q}\, f^\star(y)} \d x \d y
\end{equation}
if and only if $K$ is a centered ellipsoid and $f$ is a translate of $f^\star$.
\end{lemma}

\goodbreak
\begin{proof}
    Writing 
    \begin{align}
    \|z\|_K^{q}
    &= \int_0^\infty \chi_{[s,\infty)}(\|z\|_K^q) \d s\\
    &= \int_0^\infty \chi_{[s^{1/q},\infty)}(\|z\|_K) \d s\\
    &= \int_0^\infty k_s(z) \d s
    \end{align} 
    where $k_s(z) = \chi_{s^{1/q}(\R^n\backslash K)}(z)$, and using the layer-cake formula \eqref{eq_layer_cake} for $f$,
    we obtain
    \begin{align}
	&\int_{\R^n}\int_{\R^n}{f(x) {\Vert x-y\Vert_K^q}  f(y)} \d x \d y \\
    &= \int_0^\infty \int_0^\infty\int_0^\infty \int_{\R^n}\int_{\R^n} \chi_{\{f\geq r\}}(x)  k_s(x-y) \chi_{\{f\geq t\}}(x) \d x \d y \d r \d s\d t.
	\end{align}
	The Riesz rearrangement inequality, Theorem~\ref{thm_BLL},  implies that
    \begin{align}
   \int_{\R^n}\int_{\R^n} \chi_{\{f\geq r\}}(x)  &k_s(x-y) \chi_{\{f \geq  t\}}(y) \d x \d y\\
	&= \int_{\R^n}\int_{\R^n} \chi_{\{f \geq r\}}(x) (1- \chi_{s^{1/q}K}(x-y)) \chi_{\{f \geq  t\}}(y) \d x \d y\\
	& =\int_{\R^n}\int_{\R^n} \chi_{\{f\geq  r\}}(x)\chi_{\{f\geq t\}}(y) \d x \d y 
	\\&\phantom{=\,}- \int_{\R^n}\int_{\R^n}  \chi_{\{f\geq  r\}}(x) \chi_{s^{1/q}K}(x-y) \chi_{\{f \geq  t\}}(y) \d x \d y\\
   & \ge
    \int_{\R^n}\int_{\R^n} \chi_{\{f\geq r\}^\star}(x)  k_s^\star(x-y) \chi_{\{f\geq  t\}^\star}(y) \d x \d y
    \end{align}
    for $ r, s,  t>0$. Note that $\int_{\R^n} \chi_{\{f\geq  r\}}(x) \d x <\infty$ for $r>0$, as $f\in L^p(\R^n)$. If there is equality in \eqref{eq_riesz_equality}, then there is a null set $M\subset(0,\infty)^3$ such that  
    \begin{equation}\label{eq_BLLp2}
    \begin{aligned}
    \int_{\R^n} \int_{\R^n}   &\chi_{\{f \ge r\}}(x)  \chi_{s^{ 1/q}K}(x-y) \chi_{\{f \ge  t\}}(y) \d x \d y  \\
    &=
    \int_{\R^n} \int_{\R^n}  \chi_{\{f \ge   r\}^\star}(x) \chi_{s^{ 1/q}K^\star}(x-y) \chi_{\{f\ge  t\}^\star}(y) \d x \d y
    \end{aligned}
    \end{equation}
for  $( r,s,  t)\in (0,\infty)^3\backslash M$.

\goodbreak    
    For almost every $( r,  t)\in(0,\infty)^2$, we have $(r,s, t)\in (0,\infty)^3\backslash M$ for almost every $s>0$.
    For  such $( r,  t)$ with $ r\ge  t$ and $s>0$ sufficiently small, the assumptions of Theorem \ref{thm_burchard} are fulfilled and therefore there are a centered ellipsoid $D$ and $a, b\in\R^n$ (depending on $( r, s, t)$)  such that
    \begin{align}
     \{f \ge  r\}= a + \alpha D,\quad s^{1/q} K=  b + \beta D,\quad \{f \ge  t\}= c + \gamma D
    \end{align}
    where $c= a+ b$. 
    Since $ K=  s^{-1/q}b + (\vol{K}/\vol{D})^{1/n} D$, the centered ellipsoid $D$ does not depend on $(r,s,t)$ and hence the vectors $a$ and $c$ do not depend on $s$. It follows that $b=0$ and  that $K$ is a multiple of $D$. Hence $a=c$ is a constant vector, which concludes the proof.
 \end{proof}

\goodbreak
\begin{lemma}\label{lem_reverse_support_Riesz}
Let  $\alpha>n$ and $K\subset\R^n$ be star-shaped with measurable radial function. If  $
f:\R^n\to[0,\infty)$ is non-zero and measurable
 and
\begin{equation}
\int_{\R^n}\int_{\R^n} \frac{f(x) f(y)}{\Vert x-y\Vert_K^{n-\alpha}} \d x \d y<\infty,
\end{equation}
then
\begin{equation}
\tilde V_\alpha(K, \hls\alpha f) \geq \tilde V_\alpha(K^\star, \hls\alpha f^\star).
\end{equation}
For $\vol{K}>0$, there is equality if and only if $K$ is a centered ellipsoid and $f$ is a translate of $f^\star$.
\end{lemma}

\goodbreak
\begin{proof}
By \eqref{eq_new_def_c} and Lemma \ref{lem_reverse_strict_rearrangement}, we have
\[\tilde V_\alpha(K, \hls\alpha f) \geq \tilde V_\alpha(K^\star, \hls\alpha f^\star).\]
By Lemma \ref{lem_reverse_strict_rearrangement}, there is equality if and only if $K$ is a centered ellipsoid and $f$ is a translate of $f^\star$.
\end{proof}

\begin{lemma}\label{lem_new_rearrangement2}
Let $\alpha>n$ and $f:\R^n\to [0,\infty)$ measurable. If 
$\,\vol{\hls \alpha f}<\infty$, then
\begin{equation}
\vol{\hls \alpha f}\ge \vol{\hls \alpha f^\star}.
\end{equation}
There is equality if and only if $f$ is a translate of $f^\star$. 
\end{lemma}
\begin{proof} By Lemma \ref{lem_reverse_support_Riesz} with $K=\hls \alpha f$ and the dual mixed volume inequality \eqref{eq_mixedvolumere} for $\alpha>n$,  we have
\begin{align}
	\vol{\hls\alpha f}
	&= \tilde V_\alpha(\hls \alpha f, \hls\alpha f)\\ 
	&\geq \tilde V_\alpha((\hls\alpha f)^\star, \hls\alpha f^\star)\\
	&\geq \vol{(\hls\alpha f)^\star}^{1-\frac {\alpha}n} \vol{\hls\alpha f^\star}^{\frac \alpha n}\\
	&= \vol{\hls \alpha f}^{1-\frac {\alpha}n} \vol{\hls\alpha f^\star}^{\frac \alpha n}.
\end{align}
The equality case follows from Lemma \ref{lem_reverse_support_Riesz}. 
\end{proof}

We are now in the position to prove affine HLS inequalities for $\alpha>n$.

\begin{theorem}\label{thm_aHLS2}
For $\alpha>n$ and non-negative $f\in L^{{2n}/({n+\alpha})}(\R^n)$,
\begin{equation}
\gamma_{n,\alpha} \Vert f\Vert_{\frac{2n}{n+\alpha}}^2
\le
n \omega_n^{{1-\frac \alpha n}} 
\vol{\hls \alpha f}^{\frac \alpha n}
\\
\le
\int_{\R^n}\int_{\R^n} \frac{f(x) f(y)}{\vert x-y\vert^{n-\alpha}} \d x\d y.
\end{equation}
There is equality in the first inequality precisely if $f(x) = a (1 +|\phi (x-x_0)|^2)^{-(n-\alpha)/2}$ for $x\in\R^n$ with $a\ge 0$, $\phi\in\gln$ and $x_0\in\R^n$.
There is equality in the second inequality if $f$ is radially symmetric.
\end{theorem}

\begin{proof}
For the first inequality, we may assume that $\vol{\hls \alpha f}$ is finite. Since $\|f\|_p=\|f^\star\|_p$, we obtain by the HLS inequality \eqref{eq_HLS2}, by \eqref{eq_new_def_c} and by Lemma \ref{lem_new_rearrangement2} that
\begin{align}
	\gamma_{n,\alpha} \|f\|_p^2 
	&\leq \int_{\R^n}\int_{\R^n} \frac{f^\star(x)f^\star(y)}{|x-y|^{n-\alpha}} \d x \d y \\
	&= n \tilde V_\alpha (\B, \hls \alpha f^\star)\\
	&= n \omega_n^{{1-\frac \alpha n}} \vol{\hls\alpha f^\star}^{\frac\alpha n}\\
	&\leq n \omega_n^{{1-\frac \alpha n}} \vol{\hls\alpha f}^{\frac\alpha n}.
\end{align}
If there is equality throughout, then $f^\star$ realizes equality in the HLS inequality \eqref{eq_HLS2}. Hence $f^\star(x) = a (1 + \lambda\, | x|^2)^{-n/p}$ for some $a\ge0$ and $\lambda>0$.
By Lemma~\ref{lem_new_rearrangement2}, we obtain the equality case.

\goodbreak
For the second inequality, assume that
\[\int_{\R^n}\int_{\R^n} \frac{f(x)f(y)}{\vert x-y\vert^{n-\alpha}}\d x\d y<\infty.\]
We set $K=\B$ in \eqref{eq_new_def_c} and apply the dual mixed volume inequality \eqref{eq_mixedvolumere} to obtain
\begin{align}
    \int_{\R^n}\int_{\R^n} \frac{f(x)f(y)}{\vert x-y\vert^{n-\alpha}}\d x\d y
    &= n \tilde V_{\alpha}(\B, \hls \alpha f)
    \geq n \omega_n^{1-\frac{\alpha} n} \vol{\hls \alpha f}^{\frac\alpha n}.
\end{align}
There is equality precisely if $\hls \alpha f$ is a ball, which is the case for radially symmetric functions.
\end{proof}

Next, we state a sharp reverse of the first inequality from Theorem \ref{thm_aHLS2} for log-concave functions.

\begin{theorem}\label{thm_reverseineq2}
For $\,\alpha>n$ and  log-concave $f\in L^2(\R^n)$, 
\[ \frac{\Gamma(n+1)^\frac{\alpha}n}{\Gamma(\alpha)}\,\vol{\hls\alpha f}^\frac{\alpha}n \leq  \|f\|_2^{2-\frac{2\alpha}n} \|f\|_1^{\frac{2\alpha}n} \leq \Vert f\Vert_{\frac{2n}{n+\alpha}}^2.\]
There is equality in the first inequality if $f(x)=a\,e^{-\Vert x - x_0 \Vert_\Delta}$ for $x\in\R^n$ with $a\ge 0, x_0 \in \R^n$ and $\Delta$ an $n$-dimensional simplex having a vertex at the origin.
\end{theorem}

\noindent
The proof of this result will be given in Section \ref{sec_reverse}.

\goodbreak
\section{Radial Mean Bodies}\label{sec_GZ}

Let $E\subset \R^n$ be a convex body. For $\alpha>-1$ and $\alpha\ne 0$, Gardner and Zhang \cite{GZ} defined the radial $\alpha$-th mean body of $E$, by its radial function for $\xi\in\sn$, as
\begin{equation}\label{def_ralpha}
\rho_{\omr \alpha E}(\xi)^\alpha = \frac 1{\vol{E}}\int_E \rho_{E-x}(\xi)^\alpha \d x
\end{equation}
for $\alpha\ne0$ and as
\[
\log(\rho_{\omr 0 E}(\xi)) = \frac 1{\vol{E}}\int_E \log(\rho_{E-x}(\xi)) \d x.
\]
They showed that $\omr p E$ is a star body for $\alpha>-1$ and a convex body for $\alpha\ge0$. 
This also follows from Proposition \ref{prop_convex} and equation \eqref{eq_gz_conn}, which we will establish below. Gardner and Zhang \cite{GZ} also showed that for $\alpha>-1$ and $\xi\in\sn$,
\begin{equation}
\rho_{\omr \alpha E}(\xi)^\alpha =\frac1{(\alpha+1)\vol{E}}\int_{E\vert \xi^\perp} \vol{E\cap (\ell_\xi+y)}_1^{\alpha+1}\d y,
\end{equation}
where $\ell_\xi=\{t\xi: t\in\R\}$ is the line in direction $\xi$ and $\vol{\cdot}_1$ denotes one-dimensional volume while $E\vert \xi^\perp$ is the image of the orthogonal projection of $E$ to the hyperplane orthogonal to $\xi$.

If $\alpha > 0$, then the definition of $\hls \alpha \chi_E$ implies that 
\begin{align}
	\rho_{\hls \alpha \chi_E}(\xi)^\alpha
	&= \int_0^\infty t^{\alpha -1} \vol{E \cap (E+t\xi)}_1 \d t\\
	&= \int_0^\infty t^{\alpha -1} \int_{E\vert \xi^\perp} \pl{(\vol{E\cap(\ell_\xi+y)}_1-t)} \d y\d t\\
	&=\int_{E\vert \xi^\perp}  \int_0^\infty t^{\alpha -1} \pl{(\vol{E\cap(\ell_\xi+y)}_1-t)} \d t \d y\\
	&=  \int_{E\vert \xi^\perp} \int_0^{\vol{E\cap(\ell_\xi+y)}_1} t^{\alpha -1} (\vol{E\cap(\ell_\xi+y)}_1-t) \d t \d y\\
	&= \frac 1{\alpha(\alpha+1)} \int_{E\vert \xi^\perp}\vol{E\cap(\ell_\xi+y)}_1^{\alpha+1} \d y.
\end{align}
Hence, 
\begin{equation}\label{eq_gz_conn}
\hls \alpha \chi_E = \Big(\frac{\vol{E}}{\alpha}\Big)^{1/\alpha} \omr \alpha E
\end{equation}
for $\alpha>0$.
If $-1<\alpha< 0$, then, using \eqref{eq_defpp}, we obtain that
\begin{align}
	\rho_{\Pi_2^{\ast,-\alpha/2} \chi_E}(\xi)^\alpha
	&= \int_0^\infty t^{\alpha -1} \vol{E \Delta (E+t\xi)} \d t\\
	&= \int_0^\infty t^{\alpha -1} \int_{E\vert \xi^\perp} 2 \min\{\vol{E\cap(\ell_\xi+y)}_1,t\} \d y\d t\\
	&=2\int_{E\vert \xi^\perp}  \int_0^\infty t^{\alpha -1} \min\{\vol{E\cap(\ell_\xi+y)}_1,t\} \d t \d y\\
	&=  2\int_{E\vert \xi^\perp} \int_0^{\vol{E\cap(\ell_\xi+y)}_1} t^\alpha \d t + \int_{\vol{E\cap(\ell_\xi+y)}_1}^\infty \vol{E\cap(\ell_\xi+y)}_1\, t^{\alpha-1} \d t \d y\\
	&= -\frac 2{\alpha(\alpha+1)} \int_{E\vert \xi^\perp}\vol{E\cap(\ell_\xi+y)}_1^{\alpha+1} \d y,
\end{align}
where $E\Delta F$ is the symmetric difference of $E,F\subset\R^n$. Hence, we obtain that
\begin{equation}\label{eq_gz_conn_2}
\Pi_2^{\ast,-\alpha/2}  \chi_E = \Big(\frac{2\vol{E}}{-\alpha}\Big)^{1/\alpha} \omr \alpha E
\end{equation}
for $-1<\alpha< 0$.

\goodbreak
See \cite{GZ,HaddadLudwig_fracsob} for information on sharp affine isoperimetric inequalities for radial mean bodies.

\goodbreak
\section{Reverse Affine HLS and Fractional $L^2$ Sobolev Inequalities}\label{sec_reverse}

We prove Theorem \ref{thm_reverseineq} and Theorem \ref{thm_reverseineq2} for log-concave functions and derive results 
for $s$-concave functions for $s> 0$. In addition, we establish reverse affine fractional $L^2$ Sobolev inequalities.

\subsection{An auxiliary result}

Let $\omega\colon[0,\infty) \to [0,\infty)$ be decreasing with 
$$0<\int_0^{\infty } t^{\alpha-1} \omega (t ) \d t<\infty$$ 
for every $\alpha>0$. For any $t_0 > 0$ and $\alpha>0$, we have
\begin{equation}
    \label{eq_analyticcontinuation}
    \int_0^{\infty } t^{\alpha-1} \omega(t) \d t=\int_{ t_0}^{\infty } t^{\alpha-1} \omega(t) \d t-\int_0^{ t_0} t^{\alpha-1} (\omega(0)-\omega(t)) \d t+ \omega(0) \frac{ t_0^\alpha}{\alpha}.
\end{equation}
If, in addition, 
\begin{equation}
    \label{eq_growthcondition_negative}
    \int_0^\infty t^{\alpha-1} (\omega(0) - \omega(t))\d t<\infty
\end{equation}
for every $\alpha \in (-1,0)$, then the right side of \eqref{eq_analyticcontinuation} is finite for $\alpha \in (-1,0)$. Moreover, it is equal to
\[\int_0^\infty t^{\alpha-1} (\omega(t)-\omega(0)) \d t,\]
which is the analytic continuation of $\alpha \mapsto \int_0^\infty t^{\alpha-1} \omega(t) \d t$ to $(-1,0)$ (see, for example, \cite[Section 1.3]{GelfandShilov}).

For $\omega(t)=e^{-t}$ and $\omega(t) = \pl{(1-s t)}^{1/s}$ with $s>0$, we obtain the well-known analytic continuation formulas for the gamma and beta functions, 
\begin{equation}
\label{eq_gammacontinuation}
\Gamma(\alpha) =  
\begin{cases}
 \int_0^{\infty } t^{\alpha-1} e^{-t} \d t &\text{for } 0<\alpha, \\[4pt]
 \int_0^{\infty } t^{\alpha-1}  (e^{-t}-1) \d t &\text{for } -1<\alpha<0,
\end{cases}
\end{equation}
and
\begin{equation}
\label{eq_betacontinuation}
s^{-\alpha} \Beta(\alpha,1+\frac 1s) = 
\begin{cases}
 \int_0^{\infty } t^{\alpha-1} \pl{(1-s t)}^{1/s} \d t &\text{for } 0<\alpha, \\[4pt]
 \int_0^{\infty } t^{\alpha-1} (\pl{(1-s t)}^{1/s}-1) \d t &\text{for } -1<\alpha<0.
\end{cases}
\end{equation}

We require the following result. It is a generalization (from $\alpha>0$ to $\alpha>-1$) of \cite[Lemma~2.6]{Milman:Pajor}, which is, according to \cite{Milman:Pajor}, a consequence of results from \cite{MOP1967}. We include the proof of the case $\alpha>0$ for the convenience of the reader. Special cases of the following lemma were obtained by Koldobsky, Pajor, and Yaskin \cite[Proof of Lemma 3.3]{KoldobskyPajorYaskin} and Fradelizi, Li, and Madiman \cite[Theorem~6.1]{FradeliziLiMadiman}.

\begin{lemma}
\label{lem_extendedMilmanPajor}
Let $\omega\colon[0,\infty) \to [0,\infty)$ be decreasing with 
$$0<\int_0^{\infty } t^{\alpha-1} \omega (t ) \d t<\infty$$ 
for every $\alpha>0$ and
$$0<\int_0^{\infty } t^{\alpha-1} (\omega (0) - \omega(t)) \d t<\infty$$ 
for every $-1<\alpha<0$.
If $\varphi\colon[0,\infty) \to [0,\infty)$ is non-zero, with $\varphi(0) = 0$, and such that $t\mapsto\varphi(t)$ and $t\mapsto\varphi(t)/t$ are increasing on $(0,\infty)$, 
then
\begin{equation}
\label{eq_zeta}
\zeta(\alpha)=
\begin{cases}
\displaystyle\left(\frac{\int_0^{\infty } t^{\alpha-1} \omega (\varphi(t) ) \d t}{\int_0^{\infty } t^{\alpha-1} \omega (t ) \d t}\right)^{\frac{1}{\alpha}}&\text{for } \alpha>0\\[12pt]
\displaystyle\exp \Big(\int_0^{\infty } \frac{\omega (\varphi(t) ) - \omega(t) }{t \,\omega(0)} \d t\Big)&\text{for } \alpha=0\\[8pt]
\displaystyle \left(\frac{\int_0^{\infty } t^{\alpha-1} (\omega(\varphi(t)) - \omega(0) ) \d t}{\int_0^{\infty } t^{\alpha-1} (\omega(t) - \omega(0)) \d t}\right)^{\frac{1}{\alpha}}&\text{for } -1<\alpha<0
\end{cases}
\end{equation}
is a continuous, decreasing function of $\alpha$ on $(-1, \infty)$. Moreover, $\zeta$ is constant  on $(-1, \infty)$ if $\varphi(t) = \lambda t$ on $[0,\infty)$ for some $\lambda>0$.
\end{lemma}
\begin{proof}
First, we show that $\zeta$ is well-defined.
Take $t_0 > 0$ such that $\varphi(t_0) > 0$, and set $v_0 = \varphi(t_0)/{t_0}$. Since $t\mapsto\varphi(t)/t$ is increasing, we have $\varphi(t) \leq v_0 t$ for $0 \leq t \leq t_0$ and $\varphi(t) \geq v_0 t$ for $t \geq t_0$.

For $\alpha \in (0, \infty)$, we have
\begin{align}
\begin{split}
\label{eq_extendedMilmanPajor_welldef_1}
v_0 \int_{t_0}^\infty t^{\alpha-1} \omega(\varphi(t)) \d t
&\leq v_0 \int_{t_0}^\infty t^{\alpha-1} \omega(v_0 t) \d t\\
&= \int_{v_0}^\infty \left(\frac t{v_0}\right)^{\alpha-1} \omega(t) \d t,
\end{split}
\end{align}
where the last integral is finite and increasing with respect to $\alpha$, and for $\alpha \in (-1,0)$,
\begin{align}
\begin{split}
\label{eq_extendedMilmanPajor_welldef_2}
v_0 \int_0^{t_0} t^{\alpha-1} (\omega(0) - \omega(\varphi(t))) \d t
&\leq v_0 \int_0^{t_0} t^{\alpha-1} (\omega(0) - \omega(v_0 t)) \d t\\
&= \int_0^{v_0} \left(\frac t{v_0}\right)^{\alpha-1} (\omega(0) - \omega(t)) \d t,
\end{split}
\end{align}
where the last integral is finite and decreasing with respect to $\alpha$.
Now \eqref{eq_extendedMilmanPajor_welldef_1}, \eqref{eq_extendedMilmanPajor_welldef_2}, and \eqref{eq_analyticcontinuation} together with the dominated convergence theorem show that $\zeta$ is well-defined and continuous on $(-1, 0)$ and $(0, \infty)$. 

For $\alpha = 0$, we have
\begin{align}
\begin{split}
\label{eq_extendedMilmanPajor_welldef_3}
\int_0^{t_0} &\Big|\frac{\omega(\varphi(t)) - \omega(t)}{t}\Big| \d t\\
&\leq \int_0^{t_0} t^{-1} (\omega(0) - \omega(\varphi(t))) \d t + \int_0^{t_0} t^{-1} (\omega(0) - \omega(t)) \d t \\
&\leq \int_0^{v_0} t^{-1} (\omega(0) - \omega(t)) \d t + \int_0^{t_0} t^{-1} (\omega(0) - \omega(t)) \d t
\end{split}
\end{align}
and
\begin{align}
\begin{split}
\label{eq_extendedMilmanPajor_welldef_4}
\int_{t_0}^\infty \left|\frac{\omega(\varphi(t))- \omega(t)}{t}\right| \d t
&\leq \int_{v_0}^\infty t^{-1} \omega(t) \d t + \int_{t_0}^\infty t^{-1} \omega(t) \d t.
\end{split}
\end{align}
The monotonicity of the last integrals of \eqref{eq_extendedMilmanPajor_welldef_1} and \eqref{eq_extendedMilmanPajor_welldef_2} shows that the integrals in \eqref{eq_extendedMilmanPajor_welldef_3} and \eqref{eq_extendedMilmanPajor_welldef_4} are finite. Hence $\zeta(0)$ is well-defined.

Next, we show that 
$\zeta$ is decreasing on $(0,\infty)$. The argument is taken from \cite[Lemma~2.6]{Milman:Pajor}.
Let $\alpha>0$. It follows from the definition of $\zeta(\alpha)$ that
\begin{equation}\label{eq_equality}
    \int_0^\infty t^{\alpha-1} \omega\big(\frac t{\zeta(\alpha)}\big)\d t = \int_0^\infty t^{\alpha-1} \omega(\varphi(t))\d t.
\end{equation}
Set
\begin{equation}
    \eta(s)=\int_s^\infty t^{\alpha-1} \big(\omega\big(\frac t{\zeta(\alpha)}\big)-\omega(\varphi(t))\big)\d t,
\end{equation}
and note that \eqref{eq_equality} implies that $\eta(0)=0$ and $\lim_{s\to\infty} \eta(s)=0$. It is clear that $\omega\big(\frac t{\zeta(\alpha)}\big)-\omega(\varphi(t))$ cannot be always positive or always negative. Since $\omega$ is decreasing and $t\mapsto \varphi(t)/t$ is increasing, the function 
$t\mapsto \omega(t/ \zeta(\alpha))- \omega(t( \varphi(t)/t))$
is first non-positive and then non-negative. Hence, $\eta(s)\ge 0$ for $s\ge 0$, that is,
\begin{equation}
\int_s^\infty t^{\alpha-1} \omega\big(\frac t{\zeta(\alpha)}\big)\d t \ge \int_s^\infty t^{\alpha-1} \omega(\varphi(t))\d t.
\end{equation}
Integrating this inequality for $0<\alpha< \beta$ gives by Fubini's theorem that
\begin{equation}\label{eq_increasing}
    \begin{split}
        \int_0^\infty t^{\beta-1} \omega(\varphi(t)) \d t
        &= \int_0^\infty \int_0^t (\beta -\alpha) s^{\beta-\alpha-1} \d s \ 
        t^{\alpha-1}\omega(\varphi(t))\d t\\
        &= (\beta -\alpha) \int_0^\infty s^{\beta-\alpha-1}\Big(\int_s^\infty t^{\alpha-1}\omega(\varphi(t))\d t\Big)\d s \\
        &\le (\beta-\alpha) \int_0^\infty s^{\beta-\alpha-1} \Big( \int_s^\infty t^{\alpha-1} \omega\big(\frac t{\zeta(\alpha)}\big)\d t\Big) \d s\\
        &\le \int_0^\infty t^{\beta-1} \omega\big(\frac t{\zeta(\alpha)}\big) \d t\\
        &= \zeta(\alpha)^{-\beta} \int_0^\infty t^{\beta-1} \omega(t)\d t,
     \end{split}
\end{equation}
which concludes the proof that $\zeta$ is increasing on $(0, \infty)$.

Next, we show that $\zeta$ is decreasing on $(-1,0)$.
By the change of variables $r = t^{-1}$, we get
\[
\zeta(\alpha)^{-1}=\left(\frac{\int_0^{\infty} r^{-\alpha-1} (\omega(0)-\omega(\psi (r)^{-1})) \d r}{\int_0^{\infty }  r^{-\alpha-1} (\omega(0)-\omega(r^{-1}))  \d r}\right)^{-\frac{1}{\alpha}}
\]
where $\psi(r)=\varphi(r^{-1})^{-1}$. Observe that $\omega(0)-\omega(r^{-1})$ is decreasing and non-negative, and that ${\psi(r)}/r = \left({\varphi(r^{-1})}/{r^{-1}}\right)^{-1}$ is increasing.
Since $-\alpha \in (0,1)$ and $\zeta$ is increasing on $(0,\infty)$ by \eqref{eq_increasing}, we obtain that $\zeta(\alpha)^{-1}$ is an increasing function of $\alpha$.

It only remains to show that $\zeta$ is continuous at $\alpha = 0$. Using \eqref{eq_analyticcontinuation} and the elementary relation,
\[
\lim_{\alpha\to 0^{\pm }} \left(\frac{a(\alpha)+\frac{c}{\alpha}}{b(\alpha)+\frac{c}{\alpha}}\right)^{\frac 1{\alpha}} =
\exp\left(\frac{a(0) - b(0)}{c}\right),
\]
which is valid for a constant $c\ne0$ and continuous functions $\alpha\mapsto a(\alpha)$ and $\alpha\mapsto b(\alpha)$, we obtain that
\begin{equation*}
    \label{eq_limitG}
\lim_{\alpha\to 0^{\pm }} \zeta(\alpha)=\exp \left(\int_0^{\infty } \frac{\omega (\varphi(t) ) - \omega(t) }{t \,\omega(0)} \d t\right).
\end{equation*}
So $\zeta$ is continuous at $\alpha=0$.
\end{proof}

\subsection{Radial mean bodies of functions}
Let $f\in L^2(\R^n)$ be non-zero and non-negative. 
We define 
\begin{equation}\label{eq_omr}
\omr \alpha f =
   \left(\frac{ \alpha}{\Vert f\Vert_2^2}\right)^\frac1\alpha \hls \alpha f
\end{equation}
for $\alpha>0$ and 
\begin{equation}\label{eq_omr2}
\omr \alpha f =
   \left(\frac{\vert \alpha\vert}{2\Vert f\Vert_2^2}\right)^\frac1\alpha \Pi_2^{\ast,-\alpha/2} f
\end{equation}
for $-1<\alpha<0$. 
In addition, we define $\omr0 f$ by its radial function for $\xi\in\sn$ as
\begin{equation}
\log(\rho_{\omr 0 f}(\xi))=  -\gamma + \int_0^\infty \frac 1t \left( \frac 1{\|f\|_2^2} \int_{\R^n} f(x) f(x+t \xi) \d x - e^{-t} \right) \d t,
\end{equation}
where $\gamma$ is Euler's constant. The definitions \eqref{eq_omr} and \eqref{eq_omr2} are compatible with \eqref{eq_gz_conn} and \eqref{eq_gz_conn_2} for $f=\chi_E$ and a convex body $E\subset\R^n$. Note that
\eqref{eq_hlsn} implies that
\begin{equation}\label{eq_omrn}
    \vol{\omr n f} =\frac{\Vert f\Vert_1^2}{\Vert f\Vert_2^2}
\end{equation}
for non-zero $f\in L^2(\R^n)$.

\subsection{Results for log-concave functions} 
We start with a simple calculation, for which we need the following notation. 
We define the simplex $\Delta_n\subset \R^n$ as the convex hull of the origin and  the standard basis vectors $e_1, \dots, e_n$. In addition, let
$B_1^n=\{(x_1, \dots, x_n): \vert x_1\vert + \dots +\vert x_n\vert \le 1\}$
and $\R^n_+=\{(x_1, \dots,x_n)\in\R^n: x_1, \dots, x_n\ge 0\}$.

\begin{lemma} \label{lem_calc}
If $f(x)=e^{-\Vert x\Vert_{\Delta_n}}$ for $x\in\R^n$, then 
\[
    f*f^-(y)= \frac1{2^n}e^{- \Vert y\Vert_{B_1^n}}
\]
for $y\in\R^n$.
\end{lemma}

\begin{proof}
We have
\[
f(x)= \begin{cases}
e^{-(x_1+\dots+x_n)} &\text{ for } x\in \R^n_+,\\
0&\text{ for } x\not\in \R^n_+.
\end{cases}
\]
We obtain that
\begin{equation}
f*f^-(y)= e^{-(y_1+\dots+y_n)}\prod_{i=1}^n\int_{\R_+\cap (\R_+-y_i)} e^{-2 x_i}\d x_i.
\end{equation}
Using 
\begin{equation}
2\int_{\R_+\cap (\R_+-y_i)} e^{-2 x_i}\d x_i = 
\begin{cases}
1 &\text{ for } y_i\ge0\\
e^{2y_i} &\text{ for } y_i<0
\end{cases}
\end{equation}
and $\mn{I}(y)=\{1\le i\le n: y_i<0\}$,
we obtain that
\begin{align}
f*f^-(y) &= \frac1{2^n} \,e^{-(y_1+\dots+y_n) +2\sum_{i\in \mn{I}(y)}^n y_i}= \frac1{2^n}\, e^{- \Vert y\Vert_{B_1^n}}
\end{align}
for $y\in\R^n$.
\end{proof}

\goodbreak

We introduce the following notation. For $-1 < \alpha < \infty$ and $s>0$, let
\[
c_{n,\alpha,s} = (1+2/s)^{-1} ((n+2/s) \Beta(\alpha+1, n+2/s))^{-1/\alpha}.
\]
For $s=0$, the limit is
\[
c_{n,\alpha,0} = \Gamma(\alpha+1)^{-1/\alpha},
\]
and for $s = \infty$, we get
\[
c_{n,\alpha,\infty} = (n \Beta(\alpha+1, n))^{-1/\alpha}.
\]

We establish the following inclusion relation.

\begin{theorem}\label{thm_inclusion}
If  $f\in L^2(\R^n)$ is non-zero and log-concave, then 
\[
c_{n,\beta,0} \omr \beta f
\subseteq  c_{n,\alpha,0} \omr \alpha f
\]
for $-1<\alpha<\beta<\infty$.
There is equality if $f(x)=a\,e^{-\Vert x - x_0\Vert_\Delta}$ for $x\in\R^n$ with $a>0, x_0 \in \R^n$ and $\Delta$ an $n$-dimensional simplex having a vertex at the origin.
\end{theorem}

\begin{proof}
Note that $g=f*f^-$ is even, attains a maximum at $y=0$ and that $g(0)=\Vert f\Vert_2^2$. Moreover,  $g$ is log-concave by Lemma \ref{lem_GG}. 

We fix $\xi \in \sn$ and see that $t\mapsto g(t \xi)$ is a positive, decreasing, log-concave function. We apply Lemma \ref{lem_extendedMilmanPajor} to $\omega(t) = g(0) e^{-t}, \varphi(t)=-\log(g(t\xi)/g(0))$. Formula \eqref{eq_zeta} becomes

\[
    \zeta(\alpha) 
    = \left(\frac{\int_0^\infty t^{\alpha-1}g(t\xi) \d t}{\int_0^\infty t^{\alpha-1} g(0) e^{-t} \d t} \right)^{1/\alpha}
    = \frac{\rho_{\omr \alpha f}(\xi)}{\Gamma(\alpha+1)^{1/\alpha}}
\]
for $\alpha>0$, where we used definitions \eqref{eq_defhls} and \eqref{eq_omr}, and
\begin{align}
    \zeta(\alpha) 
    &= \left(\frac{\int_0^\infty t^{\alpha-1}( g(t\xi)-g(0)) \d t}{g(0)\int_0^\infty t^{\alpha-1} (e^{-t}-1) \d t} \right)^{1/\alpha}\\
    &= \left(\frac{\int_0^\infty t^{\alpha-1}\big(\int_{\R^n} (f(x)\,f(x+t\xi)-f^2(x))\d x\big) \d t}{g(0)\int_0^\infty t^{\alpha-1} (e^{-t}-1) \d t} \right)^{1/\alpha}\\
      &= \left(-\frac{\int_0^\infty t^{\alpha-1}\big(\int_{\R^n} (f(x+t\xi)-f(x))^2\d x\big) \d t}{2\,g(0)\,\Gamma(\alpha)} \right)^{1/\alpha}\\
      &= \frac{\rho_{\omr \alpha f}(\xi)}{\Gamma(\alpha+1)^{1/\alpha}}
\end{align}
for $-1<\alpha<0$, where we used  \eqref{eq_gammacontinuation} to obtain the gamma function and definitions \eqref{eq_defpp} and \eqref{eq_omr2}.
This proves the inclusion. There is equality in Lemma \ref{lem_extendedMilmanPajor}
 if $g(t \xi) = e^{- h(\xi) t}$ for some function $h: \sn\to (0,\infty)$. Hence, the equality case follows from Lemma \ref{lem_calc} and the $\sln$ and  translation invariance and homogeneity of $\vol{\omr \alpha f}$.
\end{proof}

Using \eqref{eq_omr} and \eqref{eq_omrn}, we obtain both Theorem \ref{thm_reverseineq} and Theorem \ref{thm_reverseineq2} immediately from the above result. Using \eqref{eq_omr2}, \eqref{eq_omrn}, and H\"older's inequality, we also obtain the following inequality from Theorem \ref{thm_inclusion}.

\begin{theorem}\label{thm_reverseineq3}
For $\,0<\alpha<1/2$ and  log-concave $f\in L^2(\R^n)$, 
\[
\alpha (c_{n,n,0}/c_{n,-2\alpha,0})^{2\alpha}\,\vol{\Pi_2^{\ast,\alpha} f}^{-\frac{2\alpha}n} \leq  \|f\|_2^{2+\frac{4\alpha}n} \|f\|_1^{-\frac{4\alpha}n} \leq \Vert f\Vert_{\frac{2n}{n-2\alpha}}^2. \]
There is equality in the first inequality if $f(x)=a\,e^{-\Vert x - x_0\Vert_\Delta}$ for $x\in\R^n$ with $a\ge 0, x_0 \in \R^n$ and $\Delta$ an $n$-dimensional simplex having a vertex at the origin. 
\end{theorem}

\noindent
The inequality in the previous theorem is a reverse inequality to the affine fractional $L^2$ Sobolev inequality \eqref{eq_L2fracsob}.

\goodbreak
\subsection{Results for $s$-concave functions}
We obtain the following inclusion relation.
\begin{theorem}\label{thm_inclusion_s}
Let  $s>0$. 
If  $f\in L^2(\R^n)$ is non-zero and $s$-concave, then
\[c_{n,\beta,s} \omr \beta f\subseteq  
c_{n,\alpha,s} \omr \alpha f
\] 
for $-1<\alpha<\beta<\infty$.
\end{theorem}

\begin{proof}
As in the proof of Theorem \ref{thm_inclusion}, note that $g=f*f^-$ is even, continuous, and attains a maximum at $y=0$. For $\xi \in S^{n-1}$, it follows that $t\mapsto g(t \xi)$ is positive and decreasing. By Lemma \ref{lem_GG}, the function $g$ is $r$-concave with $r={s}/{( ns + 2)}$.

We apply Lemma \ref{lem_extendedMilmanPajor} with $\omega(t) = g(0) \pl{(1-rt)}^{1/r}$ and $\varphi(t)=(1-(g(t\xi)/g(0))^r)/r$. 
We obtain that
\[    \zeta(\alpha) 
    = \left(\frac{\int_0^\infty t^{\alpha-1}g(t\xi) \d t}{\int_0^\infty g(0) t^{\alpha-1} \pl{(1-rt)}^{1/r} \d t} \right)^{1/\alpha}
    = \frac{r \,\rho_{\omr \alpha f}(\xi)}{(\alpha \Beta(\alpha, 1+\frac1r))^{1/\alpha}}
\]
for $\alpha>0$, where we used definitions \eqref{eq_defhls} and \eqref{eq_omr}, and
\[
\zeta(\alpha) 
    = \left(\frac{\int_0^\infty t^{\alpha-1}(g(0) - g(t\xi)) \d t}{\int_0^\infty t^{\alpha-1} (g(0) \pl{(1-rt)}^{1/r} - g(0)) \d t} \right)^{1/\alpha}\\
    = \frac{r\,\rho_{\omr \alpha f}(\xi)}{(\alpha\Beta(\alpha,1+\frac1r)) ^{1/\alpha}}
\]
for $-1<\alpha<0$, where we used formula \eqref{eq_betacontinuation} and the definitions \eqref{eq_defpp} and \eqref{eq_omr2}. The result now follows from Lemma~\ref{lem_extendedMilmanPajor}.
\end{proof}

\goodbreak
For $s\to 0$, we recover Theorem \ref{thm_inclusion} from Theorem \ref{thm_inclusion_s}. For $s\to\infty$ and $f=\chi_E$ for a convex body $E\subset\R^n$, Theorem \ref{thm_inclusion_s} implies that
\begin{equation}\label{eq_rmb_incl}
c_{n,\beta,\infty} \omr \beta E\subseteq  
c_{n,\alpha,\infty} \omr \alpha E
\end{equation}
for $-1<\alpha<\beta$.
This recovers  Theorem 5.5  by Gardner and Zhang \cite{GZ},
who showed that there is equality in \eqref{eq_rmb_incl} precisely for $n$-dimensional simplices. The problem to determine the precise equality conditions in Theorem \ref{thm_inclusion} and Theorem~\ref{thm_inclusion_s} is open. 

It follows from the proof of Theorem \ref{thm_inclusion_s} under the assumptions given there that
\begin{equation}\label{eq_rad_inclusion}
c_{n,\beta,s}\, \rho_{\omr \beta f}(\xi)
\leq  
c_{n,\alpha,s}\, \rho_{\omr \alpha f}(\xi)
\end{equation}
with equality for $\xi\in\sn$ if $f*f^-(t\xi)= a\pl{(1-\lambda t)}^{1/s}$ for $t>0$ with $a\ge 0$ and $\lambda >0$. The following lemma shows that the inequality in \eqref{eq_rad_inclusion} is sharp in some directions and that the constants in Theorem \ref{thm_inclusion_s} are optimal.

\begin{lemma}
Let $s>0$. If $f(x) = \pl{(1-\|x\|_{\Delta_n})}^{1/s}$ for $x\in\R^n$, then 
\[
    f*f^-(y) = a \pl{\big(1-\frac 12 \|y\|_{B^n_1}\big)}^{n+2/s} 
\]
for every $y\in\R^n$ with $y_1+\dots+y_n = 0$, where $a=\Beta(n,1+2/s)/(n-1)!$.

\end{lemma}

\begin{proof} Let $\mn{t}=\max\{0,-t\}$ for $t\in\R$ and $\mn{y}=(\mn{(y_1)}, \dots, \mn{(y_n)})$ for $y\in\R^n$.  Setting $\delta(t) = \pl{(1-t)}^{1/s}$, we have
\begin{align}
f*f^-(y)
&= \int_{\R^n_+\cap (\R^n_+-y)}  \delta\big( \sum_{i=1}^n x_i \big)\, \delta\big( \sum_{i=1}^n (x_i + y_i) \big) \d x \\
&= \int_{\R^n_++\mn{y}} \delta\big( \sum_{i=1}^n x_i \big)\, \delta\big( \sum_{i=1}^n (x_i + y_i) \big) \d x \\
&= \int_{\R^n_+} \delta\big( \sum_{i=1}^n (x_i+\mn{(y_i)})\big) \,\delta\big( \sum_{i=1}^n (x_i + y_i+\mn{(y_i)} \big) \d x \\
&= \int_{\R^n_+} \delta\big( \sum_{i=1}^n x_i+\sum_{i=1}^n \mn{(y_i)} \big) \,\delta\big( \sum_{i=1}^n x_i + \sum_{i=1}^n \pl{(y_i)} ) \d x \\
&= \frac1{(n-1)!} \int_0^\infty r^{n-1} \delta( r + \|\mn{y}\|_{B^n_1} ) \,\delta( r + \|\pl{y}\|_{B^n_1} ) \d r.
\end{align}
Now note that if  $\|\mn{y}\|_{B^n_1} = \|\pl{y}\|_{B^n_1} = \frac 12 \|y\|_{B^n_1}=t$, then
\begin{align}
f*f^-(y)
&= \frac1{(n-1)!} \int_0^\infty r^{n-1} \delta( r + t)^2 \d r \\
&= \frac1{(n-1)!} \int_0^\infty r^{n-1} \pl{(1 - r - t)}^{2/s} \d r. 
\end{align}
For $t > 1$, this quantity is $0$. Otherwise, we get
\begin{align}
f*f^-(y)
&= \frac1{(n-1)!} \int_0^{1-t} r^{n-1} (1 - r - t)^{2/s} \d r \\
&= \frac1{(n-1)!} (1 - t) \int_0^1 ((1 - t)s)^{n-1} ((1 - t) - (1 - t) r)^{2/s} \d r \\
&= \frac1{(n-1)!} (1 - t)^{n+2/s} \int_0^1 r^{n-1} (1-r)^{2/s} \d r, 
\end{align}
which completes the proof.
\end{proof}

\goodbreak

Using \eqref{eq_omr}, \eqref{eq_omrn}, and H\"older's inequality, we obtain the following inequalities from Theorem \ref{thm_inclusion_s}.

\begin{corollary}\label{cor_reverseineq}
Let $s>0$. If $f\in L^2(\R^n)$ is $s$-concave on its support, then 
\[ \alpha(c_{n,\alpha,s}/c_{n,n,s})^\alpha 
\,\vol{\hls \alpha f}^\frac{\alpha}n \geq  \|f\|_2^{2-\frac{2\alpha}n} \|f\|_1^{\frac{2\alpha}n} \geq \Vert f\Vert_{\frac{2n}{n+\alpha}}^2\]
for $\,0<\alpha<n$ and the inequalities are reversed for $\alpha>n$.
\end{corollary}

\noindent
\goodbreak
Using  \eqref{eq_omr2}, \eqref{eq_omrn}, and H\"older's inequalities, we obtain the following inequality from Theorem \ref{thm_inclusion_s}.

\begin{corollary}\label{cor_reverseineq2}
Let $s>0$. If $f\in L^2(\R^n)$ is $s$-concave on its support, then 
\[ \alpha (c_{n,n,s}/c_{n,-2\alpha,s})^{2\alpha}
\,\vol{\Pi_2^{\ast,\alpha} f}^{-\frac{2\alpha}n} \leq  \|f\|_2^{2+\frac{4\alpha}n} \|f\|_1^{-\frac{4\alpha}n} \leq \Vert f\Vert_{\frac{2n}{n-2\alpha}}^2\]
for $\,0<\alpha<1/2$. 
\end{corollary}

\noindent
The following problems remain open:  Are the first inequalities in Corollary \ref{cor_reverseineq} and Corollary \ref{cor_reverseineq2} sharp for $f\not\equiv 0$?

\subsection*{Acknowledgments}
The authors thank the referees for their helpful remarks. J.~Haddad was supported by grants RYC2021-031572-I and PID2022-136320NB-I00, funded by the Ministry of Science and Innovation / State Research Agency /
10.13039 / 501100011033 and by the E.U. Next Generation EU/Recovery, Transformation and Resilience Plan.
M.~Ludwig was supported, in part, by the Austrian Science Fund (FWF) Grant-DOI:  10.55776/P34446 and Grant-DOI: 10.55776/P37030.

\end{document}